\documentclass[11pt]{amsart}
\usepackage{amscd,amssymb}
\usepackage[arrow,matrix]{xy}
\usepackage[dvips]{graphicx}
\usepackage{psfrag}

\theoremstyle{plain}
\newtheorem{thm}[subsection]{Theorem}
\newtheorem{lem}[subsection]{Lemma}
\newtheorem{prop}[subsection]{Proposition}
\newtheorem{cor}[subsection]{Corollary}

\theoremstyle{definition}
\newtheorem{rk}[subsection]{Remark}
\newtheorem{definition}[subsection]{Definition}
\newtheorem{ex}[subsection]{Example}

\numberwithin{equation}{section}
\newcommand{\F}{{\mathcal F}}
\newcommand{\A}{{\mathcal A}}

\newcommand{\CC}{{\mathcal C}}
\newcommand{\LL}{{\mathcal L}}
\newcommand{\cP}{{\mathcal P}}

\newcommand{\V}{{\mathcal V}}

\newcommand{\al}{{\alpha}}

\newcommand{\Z}{\mathbb{Z}}

\newcommand{\R}{\mathcal{R}}
\newcommand{\C}{\mathbb{C}}
\newcommand{\K}{\mathbb{K}}

\newcommand{\PP}{\mathbb{P}}

\newcommand{\T}{\mathbb{T}}

\DeclareMathOperator{\Hom}{Hom}

\DeclareMathOperator{\supp}{supp}

\begin{document}

\title [Characteristic varieties for a class of line arrangements]
{Characteristic varieties  for a class of line arrangements  }

\author[Thi Anh Thu DINH ]{ Thi Anh Thu DINH }
\address{  Laboratoire J.A. Dieudonn\'e, UMR du CNRS 6621,
                 Universit\'e de Nice Sophia Antipolis,
                 Parc Valrose,
                 06108 Nice Cedex 02,
                 FRANCE.}
\email {dinh@unice.fr}



\begin{abstract}Let $\A$ be a line arrangement in  the complex projective plane $\PP^2$, having the points of multiplicity $\geq 3$ situated on two lines in $\A$, say $H_0$ and $H_{\infty}$. Then we show that the non-local irreducible components of the first resonance variety $\R_1(\A)$ are 2-dimensional and correspond to parallelograms $\cP$ in $\C^2=\PP^2 \setminus H_{\infty}$ whose sides are in $\A$ and for which $H_0$ is a diagonal.
\end{abstract}

\maketitle

\section{Introduction } \label{s0}

Let $\A$ be a line arrangement in the complex projective plane $\PP^2$ and denote by $M$ the corresponding
arrangement complement. It is classically known that the fundamental group of $M$ is abelian if and only if
the line arrangement $\A$ has only double points, see Theorem 1.1 in \cite{CDP} and the reference to Zariski's
work given there.

The main object of study of the paper \cite{CDP} is the next simplest case of a line arrangement, namely a line 
arrangement $\A$ where the points of multiplicity $\geq 3$ are situated on a line $H_0 \in \A$. If we take this line as the line at infinity, then this is the same as studying affine line arrangements in the plane $\C^2$
having only nodes. For this class of arrangements, call it $\CC_1$, one can compute explicitely the corresponding fundamental group $\pi_1(M)$ and the {\it characteristic varieties} $\V_k(M)$, see \cite{CDP}.
The features of these characteristic varieties, which are also denoted by $\V_k(\A)$, include

\medskip

\noindent (i) there are no translated components, and

\medskip

\noindent (ii) for any irreducible component $W$ of some characteristic variety $\V_k(M)$, the dimension of
$H^1(M,\LL)$ is constant for $\LL \in W \setminus \{1\}$.

\medskip

In the present paper we show that the description of the characteristic varieties $\V_k(M)$ can be pushed one step further, namely to the class $\CC_2$ of line arrangements $\A$ where the points of multiplicity $\geq 3$ are situated on two lines in $\A$. The corresponding characteristic varieties $\V_k(M)$ enjoy the
properties (i) and (ii) above. 

The deleted $B_3$-arrangement studied by A. Suciu for its translated component, see \cite{S1}, has 
the points of multiplicity $\geq 3$ situated on three lines in $\A$. This shows that in some sense our result is the best possible.

\medskip

In fact, to determine these varieties, we use a recent result by S. Nazir and Z. Raza \cite{NR} saying that in such a situation all rank one local systems are admissible. A consequence of this fact is that the properties (i) and (ii) hold, as shown in \cite{Dadm}. Therefore, the characteristic varieties $\V_k(M)$ are completely determined for this class of arrangements by the {\it resonance varieties} $\R_k(M)$, also denoted by $\R_k(\A)$. So the main technical point of this paper is a detailed description of these resonance varieties. Roughly speaking, our main result can be stated as follows. For a more precise statement, see Theorem \ref{t3}.

\begin{thm} \label{t0}

Let $\A$ be a line arrangement in  $\PP^2$, having the points of multiplicity $\geq 3$ situated on two lines in $\A$, say $H_0$ and $H_{\infty}$. Then the non-local irreducible components of  $\R_1(M)$ are 2-dimensional and correspond to parallelograms $\cP$ in $\C^2=\PP^2 \setminus H_{\infty}$ whose sides are in $\A$ and for which $H_0$ is a diagonal.

\end{thm}
A moment thought shows that this statement is in fact symmetric with respect to the two lines  $H_0$ and $H_{\infty}$. That is, if we look for the parallelograms $\cP'$ in $\C^2=\PP^2 \setminus H_{0}$ whose sides are in $\A$ and for which $H_{\infty}$ is a diagonal, we get exactly the same family of parallelograms as
in Theorem \ref{t0}.

\begin{cor} \label{c0}

If $\A$ is a line arrangement in the classe $\CC_2$, then the Orlik-Solomon algebra $A(\A)$ determines the Tutte polynomials of the matroid associated to $\A$.

\end{cor}

This claim follows from Theorem \ref{t0} and Theorem 3.16 in Falk's paper \cite{F}.

\medskip

In the second section we recall the notion of admissible local system and  the fact that any local system of rank one on the complement of a line arrangement in the classes $\CC_1$ and $\CC_2$ is admissible.

In the third section we recall the notions of characteristic and resonance varieties, and we reobtain the description of the characteric varieties for the class $\CC_1$ given in \cite{CDP} using the new approach
described above.

In the final section we prove the main result of this paper which is the description of the characteristic and resonance varieties for the class $\CC_2$.

\section{Admissible rank one local systems} \label{s1}

Let $\A=\{H_0,H_1,...,H_n\}$ be a line arrangement in $\PP^2$ and set $M=\PP^2 \setminus (H_0 \cup...\cup H_n)$. Let $\T(M)=\Hom(\pi_1(M),\C^*)$ be the character variety of 
$M$. This is an algebraic torus
$\T(M) \simeq (\C^*)^{n}$. Consider the exponential mapping
\begin{equation} 
\label{e1}
\exp :H^1(M,\C) \to H^1(M,\C^*)=\T(M)
\end{equation}
induced by the usual exponential function $\exp(2 \pi i -): \C \to \C^*$. 

Clearly one has 
$\exp(H^1(M,\C))=\T(M)$
and  $\exp(H^1(M,\Z))=\{1\}$.

More precisely, a rank one local system $\LL\in \T(M)$ corresponds to the choice of some monodromy complex numbers
$\lambda _j \in \C^*$ for $0 \leq j \leq n$ such that $\lambda _0 ...\lambda _n=1$. And a cohomology class
$\alpha \in H^1(M,\C)$ is given by
\begin{equation} 
\label{e2}
\alpha=\sum_{j=0,n}a_j\frac {df_j}{f_j}
\end{equation}
where the residues $a_j \in \C$ satisfy $\sum_{j=0,n}a_j=0$ and $f_j=0$ a linear equation for the line $H_j$. With this notation, one has
$\exp (\alpha)=\LL$ if and only if $\lambda _j =\exp(2\pi i a_j)$ for any $j=0,...,n$.

\begin{definition} \label{d1}
A local system $\LL \in \T(M)$ as above is admissible if there is a cohomology class $\alpha \in H^1(M,\C)$ such that $\exp(\alpha)=\LL$, $a_j\notin \Z_{>0}$ for any $j$ and, for any point $p \in H_0 \cup...\cup H_n$ of multiplicity at least 3 one has
$$a(p)=\sum_ja_j \notin \Z_{>0}.$$
Here the sum is over all $j$'s such that $p \in H_j$.

\end{definition}

\begin{rk} \label{r1}
When $M$ is a hyperplane arrangement complement  one usually defines the notion of an
{\it admissible} local system $\LL$ on $M$ in terms of some conditions on the residues of an associated
logarithmic connection $\nabla(\al)$ on a good compactification of $M$, see for instance \cite{ESV},\cite{STV}, \cite{F} , \cite{LY} and \cite{DM}.
For such an admissible local system $\LL$ on $M$ one has
$$ \dim H^k(M,\LL)=\dim H^k(H^*(M,\C), \alpha \wedge)$$
for all $k$. 

For the case of line arrangements such a good compactification is obtained  by blowing-up the points of multiplicity at least 3 in $\A$. This explains the simple version of this definition given above.

\end{rk}

The following result was obtained by S. Nazir and Z. Raza \cite{NR}.

\begin{prop} \label{p0}

If $\A$ is a line arrangement in the classes $\CC_1$ or $\CC_2$, then any local system $\LL \in \T(M)$
is admissible.

\end{prop}

\section{Characteristic and resonance varieties} \label{s3}

To go further, we need the characteristic and resonance varieties, whose definition is recalled below.
The {\em characteristic varieties}\/ of $M$ 
are the jumping loci for the cohomology of $M$, with 
coefficients in rank~$1$ local systems:
\begin{equation} \label{e3}
\V^i_k(M)=\{\rho \in \T(M) \mid \dim H^i(M, \LL_{\rho})\ge k\}.
\end{equation}
When $i=1$, we use the simpler notation $\V_k(M)=\V^1_k(M)$.

The {\em resonance varieties}\/ of $M$ 
are the jumping loci for the cohomology of the complex  $H^*(H^*(M,\C), \alpha \wedge)$, namely:
\begin{equation} \label{e4}
\R^i_k(M)=\{\alpha \in H^1(M,\C) \mid \dim H^i(H^*(M,\C), \alpha \wedge)\ge k\}.
\end{equation}
When $i=1$, we use the simpler notation $\R_k(M)=\R^1_k(M)$.

\begin{ex} \label{ex0}
We consider the resonance and characteristic varieties for a central arrangement $\A$ of $n$ lines 
$H_1$,...,$H_n$ in $\C^2$. In other words, we have chosen the line $H_0$ to be the line at infinity.
Let $\omega_j=\frac {df_j}{f_j}$ ($j=1,\ldots,n$) be the canonical genarators of $H^1(M,\C)$ associated to this choice. Let $\alpha=\sum_{j=1}^n a_j\omega_j$, $\beta=\sum_{j=1}^nb_j\omega_j$ be two elements of $H^1(M,\C)$.

When $n=1$, $M$ has the homotopy type of the circle $S^1$ and hence, $\R_1(M)=\{0\}$,     $\R_k(M)=\emptyset$ for all $k>1$, $\V_1(M)=\{1\}$ and $\V_k(M)= \emptyset $ for all $k>1$.

When $n=2$, $M$ has the homotopy type of the real torus $S^1 \times S^1$. It follows that                the group $H^2(M,\C)$ is generated by $\omega_1\wedge\omega_2$ and hence $\alpha\wedge\beta=0$ iff $a_1b_2-a_2b_1=0$, i.e $\beta$ and $\alpha$ are collinear. We get  $\R_j(M)=\{0\}$ for $j=1,2$,    $\R_k(M)=\emptyset$ for all $k>2$, $\V_j(M)=\{1\}$ for $j=1,2$  and $\V_k(M)= \emptyset $ for all $k>2$.

When $n>2$, note that $\omega_1\omega_2,\ldots,\omega_1\omega_n$ form a basis of $H^2(M,\C)$. We find that, for $\alpha \ne 0$, $\dim H^i(H^*(M,\C), \alpha \wedge)$ is equal to either 0 or $n-2$, see for instance Lemma 3.1 in \cite{LY}. The later case occurs iff $\sum_{i=1}^na_i=0$  Hence, $\R_k(M)=\{\alpha|\sum_{i=1}^na_i=0\}$ for $0<k\le n-2$,   $\R_k(M)=\{0\}$ for $k=n-1,n$ and $\R_k(M)=\emptyset$ for the other $k$.

It is known that $\V_k(M)=\{\lambda \in \T(M)=(\C^*)^n ~~|~~\prod_{i=1}^n\lambda_i=1\}$ for $0<k\le n-2$,  $\V_k(M)=\{1\}$ 
for $k=n-1,n$ and $\V_k(M)=\emptyset$
for the other $k$. For a general approach to this question see Proposition 6.4.3 in \cite{D1}.
\end{ex}

\begin{ex} \label{ex1}
Here we give the description of the resonance and characteristic varieties for a nodal arrangement $\A$  in $\PP^2$. As we mentioned in the Introduction, $\pi_1(M)=\Z^n$ where $|\A|=n+1$. On the other hand, the first resonance and characteristic varieties 
depend only on the fundamantal group, \cite{DPS1}. It follows that we can replace $M$ by $(\C^*)^n$ and hence this easily yields
$\R_k(M)=0$ and $\V_k(M)=1$ for $0<k \leq n$.

\end{ex}

The more precise relation between the resonance and characteristic varieties can be summarized as follows, see  \cite{CS} or \cite{DPS1} for a more general result.

\begin{thm} \label{t1} Assume that $M$ is a hyperplane arrangement complement. Then the irreducible components $E$ of the resonance variety
$\R_1(M)$ are linear subspaces in $H^1(M,\C)$ and the exponential mapping \eqref{e1} sends these irreducible components $E$ onto the irreducible components $W$ of $\V_1(M)$ with $1 \in W$.
\end{thm}

\medskip

One has also the following result, see \cite{Dadm}, Remark 2.9 (ii).

\begin{thm} \label{t2} If any local system in $\T(M)$ is admissible then for any $k$ one has the following.

\medskip

\noindent (i) There are no translated components in the characteristic variety $\V_k(M)$ , and

\medskip

\noindent (ii) For any irreducible component $W$ of some characteristic variety $\V_k(M)$, the dimension of
$H^1(M,\LL)$ is constant for $\LL \in W \setminus \{1\}$.

\end{thm}

The following result was obtained in \cite{CDP}. We provide here an alternative simpler proof.

\begin{prop} \label{p1} (Description of the resonance varieties for the class $\CC_1$)

The irreducible components of $\R_1(\A)$ are vector subspaces $E_i$ one-to-one corresponding to the maximal families $\A_i$ of parallel lines in $\A$ with $\#\A_i>1$. In particular, dim$E_i=\#\A_i$.

\end{prop}
\begin{proof}
Keeping the notations in Example \ref{ex0} (in particular $H_0$ is the line at infinity), we have $$\alpha\wedge\beta=(\sum{a_j \omega_j})\wedge(\sum{b_j \omega_j})=\sum_{1\leq i<j\leq n}(a_i b_j - a_j b_i)\omega_i\wedge\omega_j$$
If $\{a\}=H_i\cap H_j$ and since the intersection points in the affine part are double points, the local complement of $\A$ at $a$, denoted by $M_a$, has the cohomology group $H^2(M_a,\C)$ generated by  unique element, namely by (the restriction of). Moreover, the forms $\omega_i\wedge\omega_j$ with and the forms $1\leq i<j\leq n$ and $H_i\cap H_j \ne \emptyset$ yield a basis for $H^2(M,\C)$, see \cite{OT}, Corollary 3.73. 

There are two possibilities to discuss.

\medskip

\noindent (i)  $\A$ contain no parallel lines: then $\omega_i\wedge\omega_j (1\leq i<j\leq n)$ form a basis of  $H^2(M,\C)$. It follows that $\alpha\wedge\beta= 0$ iff $$
\left|
\begin{array}{cc}
a_i & a_j \\
b_i & b_j
\end{array}
\right|=0
\mbox{ for all }
1\leq i< j\leq n.
$$
If $\alpha \ne 0$, then we can assume for instance that $a_1\neq0$. Then we have
$$
\left|
\begin{array}{cc}
a_1 & a_i \\
b_1 & b_i
\end{array}
\right|=0
\Leftrightarrow
\left(
\begin{array}{c}
b_i \\
b_1
\end{array}
\right)
=\lambda_i
\left(
\begin{array}{c}
a_i \\
a_1
\end{array}
\right)
=\frac{b_1}{a_1}
\left(
\begin{array}{c}
a_i \\
a_1
\end{array}
\right)
\Leftrightarrow
\alpha^\bot=\C\langle\alpha\rangle.
$$
Here the orthogonal complement $\alpha^\bot$ is taken with respect to the cup-product.
Therefore  $\mathcal{R}_1(\A)=\{0\}.$

\medskip

\noindent (ii) If there are $s$ families of parallel lines in $\A$: we can write $\A=\A_1\cup\ldots\cup\A_s\cup\A_{s+1}$, where $\A_1,\ldots,\A_s$ are the families of parallel lines
(containing at least 2 lines) and $\A_{s+1}$ consists of the lines which cut all the other lines in $\A$. Let $I_i$ be the index set of $\A_i$ ($\#\A_i>1$ for $i\neq s+1$). Note that if $H_i\parallel H_j$, then $\omega_i\wedge\omega_j=0$. Thus we have 
$$\alpha\wedge\beta=(\sum_{i=1}^{s+1}{\sum_{j_i\in I_i}{a_{j_i}\omega_{j_i}}})\wedge(\sum_{i=1}^{s+1}{\sum_{j_i\in I_i}{b_{j_i}\omega_{j_i}}})
=\sum_{\begin{array}{c}
p<q \\
H_p\cap H_q\neq\emptyset
\end{array}}
{\left|\begin{array}{cc}a_p & a_q \\ b_p & b_q \end{array}\right| \omega_p\wedge\omega_q}.$$
If $a_l\neq0$ for some $l\in I_{s+1}$: by considering the minors $\left|
\begin{array}{cc}
a_l & a_{j_i} \\
b_l & b_{j_i}
\end{array}\right|$, we find as above that $\alpha^\bot=\C\langle\alpha\rangle$.\

If $a_l,a_m\neq0$ for $l\in\A_i\neq\A_j\ni m$: by considering the minors $\left|
\begin{array}{cc}
a_l & a_{j_i} \\
b_l & b_{j_i}
\end{array}\right|$ et $\left|
\begin{array}{cc}
a_m & a_{j_i} \\
b_m & b_{j_i}
\end{array}\right|$, we find that $\alpha^\bot=\C\langle\alpha\rangle$.\

Thus, without losing the generality, now we can assume that $a_{j_i}=0$ for $j_i\notin I_1$. It is easy to check that $\alpha^\bot=\bigoplus_{j_1\in I_1}\C\langle\omega_{j_1}\rangle$, and this space contains strictly $\C\langle\alpha\rangle$.\

So we have 
$$\R_1(\A)=\bigcup_{i\neq s+1}E_i$$
where $E_i=\bigoplus_{j_i\in I_i}\C\langle\omega_{j_i}\rangle.$
In other words, $E_i$ consists exactly of the 1-forms $\alpha$ supported on the lines in the family $\A_i$.
\end{proof}

Using now Theorem \ref{t1} and Theorem \ref{t2}, we recover the following description of the characteristic varieties given in \cite{CDP}.

\begin{cor} \label{c1} (Description of the characteristic varieties for the class $\CC_1$)

$$\V_1(\A)=\{\lambda \in\mathbb{T}(M)\simeq(\C^*)^n~~\vert~~\exists i<s+1 \mbox{ such that }\lambda_j=1\mbox{ }\forall j\notin I_ i\}=$$
$$=(\C^*)^{|\A_1|}\times1\times\ldots\times1\cup1\times(\C^*)^{|\A_2|}\times\ldots\times1\cup\ldots\cup1\times\ldots\times1\times(\C^*)^{|\A_s|}\times\ldots  \times1.$$

\end{cor}

The last equality above holds under the assumption that the lines in $\A$, distinct from $H_0$, have been numbered such that $H_i \in \A_p$ and $H_j \in \A_q$ with $p<q$ implies $i<j$.

\begin{rk} \label{rkS} 

The groups $G=\pi_1(M)$ for line arrangements in class $\CC_1$ are of the form
         $$G= \F_{n_1} \times \cdots \times \F_{n_r}$$
         where $\F_m$ denotes the free group on $m$ generators.
Such a group $G$ is the right-angled Artin group $G_{\Gamma}$
corresponding to the complete multi-partite graph
$\Gamma=K_{n_1,\dots,n_r}$.  The resonance varieties
$\R^1_1(G_\Gamma,\K)$ of any right-angled Artin group $G$, and
over any field $\K$, were computed in \cite{PS06}.  The characteristic
varieties $\V^1_1(G_\Gamma,\C)$ of any right-angled Artin group
$G$ were computed in \cite{DPS1}.  And finally, the cohomology
jumping loci $\R^i_k(G_\Gamma, \K)$ and $\V^i_k(G_\Gamma, \K)$,
for all $i$, $k\ge 1$, and over all fields $\K$, were computed in
\cite{PS08}.

\end{rk}

\section{The resonance and the characteristic varieties for the class $\CC_2$} \label{s4}

In this section we use a slightly different notation than above.

Let $\A$ be an {\it affine} line arrangement  whose  points of multiplicity $\ge 3$ lie all on the line $H_0$. Let $\A_0,\ldots,\A_m$ denote the families of parallel lines in $\A$, with $\A^0\ni H_0$
(here $\#\A_j \ge 2$ for $j>0$). Let $\A^1,\ldots,\A^n$ denote the central subarrangements of $\A$ strictly containing $H_0$. These are the local arrangements $\A_x$ based at the multiple points $x$ of $\A$ along $H_0$.

We set $\A^0=\A_0\backslash\{H_0\}$ and note that this arrangement can be empty. 

To complete the picture, one can introduce the line at infinity $H_{\infty}$ and consider the projective arrangement $\A'$ obtained from $\A$ by adding this line. Then $\A'$ belongs to the class $\CC_2$: the points of multiplicity $\ge 3$ lying on the line at infinity $H_{\infty}$ correspond exactly to the families of parallel lines $\A_j$ with
$\#\A_j \ge 2$. However, we do not use this point of view explicitly.

\medskip

Let $\alpha=\sum_{H\in \A}a_H\omega_H$ and $\beta=\sum_{H\in \A}b_H\omega_H$ be two elements of $H^1(M(\A),\C)$. For every intersection point $x$ of the lines in $\A$, we denote 
$$\alpha_x = \sum_{x\in H} a_H\omega_H.$$
The isomorphism between the cohomology algebra $H^*(M(\A),\C)$ and the Orlik-Solomon algebra, and the decomposition of the latter via the poset of $\A$ (see \cite{OT}, Theorem 3.72) imply that $\alpha\beta=0$ iff $\alpha_x\beta_x =0$  for every $x\in L_2(\A)$. This condition is satisfied in the following cases:

\medskip

\noindent ($I$) if $\alpha_x = 0$, then $\beta_x$ can be arbitrary;

\medskip

\noindent ($II$) if $\alpha_x \ne 0$, then either:

 ($II_a$) if $x$ is a double point or $\sum_{x\in H}a_H\neq0$, then $\alpha_x$ and $\beta_x$ must be collinear; 

\medskip

 ($II_b$) if the multiplicity of $x$ is at least 3 and $\sum_{x\in H}a_H=0$, then the orthogonal 
 
 complement of $\alpha_x$ in $H^1(M(\A_x),\C)$ is $\{\beta_x~~\vert~~\sum_{x\in H}b_H =0\}$.

To get the last two claims, recall Example \ref{ex0}.

\bigskip

First we give a criterion for $\alpha$ not to belong to $\R_1(\A)$.
\begin{lem} \label{l1}  (Key Lemma)
Let $H_1\in\A^i\neq\A^j\ni H_2$, $(i,j = 0,\ldots,n)$, be two lines different from $H_0$ such that $H_1\cap H_2\neq \emptyset$ and $a_{H_1},a_{H_2}\neq0$. If there exists a line $H_3\neq H_0$ cutting $H_2$ at double point and $a_{H_3}\neq0$, then $\alpha\notin\R_1(\A)$.

\end{lem}

\begin{proof}
As a consequence  of the hypotheses, we see as in the proof of Proposition \ref{p1} that, for $\beta \in \alpha^\perp$, the quotients $\dfrac{b_{H_k}}{a_{H_k}}$ are equal for $k=1,2,3$.
We treat below the situation when $i>0$ and $j>0$. The remaining cases, i.e. $i=0$ or $j=0$ are simpler, and the reader can treat them essentially using the same approach.

\medskip
 
\noindent {\bf Case 1:} $H_3$ cuts $H_1$ at a double point.
Then for every $ H\in \A\backslash\{H_0\}$,
either 

\noindent (1a) $H\cap H_1$ is a double point: the statement ($II_a$) gives us $\alpha^\perp\subset\big{\{}\beta\big{\vert}b_H=\dfrac{b_{H_1}}{a_{H_1}}a_H\big{\}}$,
or 

\noindent (1b) $H\cap H_2$ is a double point: idem, we have $\alpha^\perp\subset\big{\{}\beta\big{\vert}b_H=\dfrac{b_{H_2}}{a_{H_2}}a_H=\dfrac{b_{H_1}}{a_{H_1}}a_H\big{\}}$\
or 

\noindent (1c) $H\cap H_3$ is a double point: idem, we have $\alpha^\perp\subset\big{\{}\beta\big{\vert}b_H=\dfrac{b_{H_3}}{a_{H_3}}a_H=\dfrac{b_{H_1}}{a_{H_1}}a_H\big{\}}$.

Let $x \in H_0$ be a point of multiplicity $\ge 3$ (such a point exists, otherwise we have an arrangement of class $\CC_1$).
If $\sum_{x\in H}a_H\neq0$, one finds that $\alpha^\perp\subset\big{\{}\beta\big{\vert}b_{H_0}=\dfrac{b_{H_1}}{a_{H_1}}a_{H_0}\big{\}}$. Otherwise, by the statement ($II_b$), for $\beta\in\alpha^\perp$, one has 
$$b_{H_0}=-\sum_{H\in\A_x\backslash\{H_0\}}b_H = -\dfrac{b_{H_1}}{a_{H_1}}\sum_{H\in\A_x\backslash\{H_0\}}a_H =\dfrac{b_{H_1}}{a_{H_1}}a_{H_0}.$$
Thus, $\beta$ is proportional to $\alpha$ and so $\alpha\notin\R_1(\A)$.

\medskip

\noindent{\bf Case 2:} If $H_3\parallel H_1$. Since every $H\in\A^i\backslash\{H_0,H_1\}$ cuts $H_3$ at double point, by the statement ($II_a$), in $\alpha^\perp$, we have $$b_H=\dfrac{b_{H_3}}{a_{H_3}}a_H=\dfrac{b_{H_1}}{a_{H_1}}a_H.$$
If $\sum_{H'\in\A^i}a_{H'}\neq0$, it is clear that we must have $b_{H_0}=\dfrac{b_{H_1}}{a_{H_1}}a_{H_0}=\dfrac{b_{H_2}}{a_{H_2}}a_{H_0}$. If this sum equals zero, the same argument as above also gives $b_{H_0}=\dfrac{b_{H_1}}{a_{H_1}}a_{H_0}$.

For $H\in\A^j\backslash\{H_0\}$,
if $H\cap H_1$ is a point, then $\alpha^\perp\subset\big{\{}\beta\big{\vert}b_H=\dfrac{b_{H_1}}{a_{H_1}}a_H=\dfrac{b_{H_2}}{a_{H_2}}a_H\big{\}}$.

If $H\parallel H_1$, using the value of the sum  $\sum_{H\in\A^j}a_H$ as above, we get $b_H=\dfrac{b_{H_2}}{a_{H_2}}a_H$.

It is easy to get the same relation for $b_H$ for the other lines $H$, since they meet either $H_1$ or $H_2$
in a double point.

\medskip

\noindent{\bf Case 3:} If $H_3\in \A^i$. Every $H\not\in\A^i$ cuts $H_1$ or $H_3$ in a double point, and one deduces that $b_H=\dfrac{b_{H_3}}{a_{H_3}}a_H$. Next, considering $\A^j$, we can prove that $b_{H_0}=\dfrac{b_{H_3}}{a_{H_3}}a_{H_0}$. Finally we deal with the lines in $\A^i$ as in the Case 2 (since there is at most one such lines not meeting $H_2$ in a double point). 

In conclusion $\alpha^\perp$ turns out to be 1-dimensional (spanned by $\alpha$) in all the cases.

\end{proof}

\begin{definition}\label{d2}
Let $\alpha$ be an element of $H^1(M,\C)$. The support of $\alpha$, denoted by $\supp \alpha$, is the set of lines $H\in\A$ with $a_H\neq0$. 
\end{definition}

Let $\alpha \in H^1(M,\C)$ be a non-zero element. We get information on the support of $\alpha$
and we decide when $\alpha \in \R_1(\A)$
in the following careful discussion of various possible situations.

\medskip

\noindent {\bf Case A:} Let us consider first the simple case where $\supp \alpha=\{H_0\}$. Whence $a_{H_0}\neq0$, we find using ($II_a$) that $\alpha^\perp\subset\{\beta ~~\vert~~ b_H =0\text{ for all } H\notin\A_0\}$. In particular $\alpha\in\R_1(\A)$ if and only if  $\#\A_0\geq 2$.

\medskip

\noindent {\bf Case B:}
 Now we assume that $\supp \alpha \ne \{H_0\}$ and take $H_1\in\supp \alpha\setminus \{H_0\}$ such that the pencil $\A^i$ which contains $H_1$ is of maximal cardinal, i.e, whenever $\#\A^j>\#\A^i$, $a_H$ must be zero for all $H$ in $\A^j\backslash\{H_0\}$. 
 
 \medskip

\noindent {\bf Case B1:   $\#\A^i>2$ and $i\neq0$.}

\medskip

\noindent {\bf Case B1a: $\supp \alpha\subset\A^i$ and $i\neq0$.}

If $\sum_{H\in\A^i}a_H=0$, then obviously $\alpha\in\R_1(\A)$.\

If $\sum_{H\in\A^i}a_H\neq0$ and $a_{H_0}\neq0$, then $b_H=\dfrac{b_{H_0}}{a_{H_0}}a_H$ for all $H\in\A$. This follows by applying $II_a$ at each of the points $x_j \in H_0$, the centers of the subarrangements $\A^j$. Since all the lines $H \in \A^0$ intersect $H_1$ at a node, we get the same result for such lines.

If $\sum_{H\in\A^i}a_H\neq0$ and $a_{H_0}=0$, it is required that $b_H=\dfrac{b_{H_1}}{a_{H_1}}a_H$ for all $H\not\parallel H_1$. If there exists another line $H_2\in\A^i$ such that $a_{H_2}\neq0$, since every line parallel to $H_1$ cuts $H_2$ at a double point, we can deduce that $\alpha\notin\R_1(\A)$.\

In the case where supp$\alpha=\{H_1\}$, obviously, dim$\alpha^\perp=\#\{$lines of the same direction with $H_1\}$.\

\medskip

\noindent {\bf Case B1b: $\supp \alpha \setminus \A^i \ne \emptyset$ and $i\neq0$.}

This means that there exists $H_3\in\A^j$ with $j \ne i$, $H_3 \ne H_0$ and $a_{H_3}\neq0$. If
 $\#(\supp \alpha \cap (\A^i\backslash\{H_0\}))\geq 3$, there must be two lines in $\supp \alpha \cap (\A^i\backslash\{H_0\})$ which cut $H_3$ at double points. So by the Key Lemma, in such a case 
 $\alpha \notin\R_1(\A)$.
 
 \medskip

\noindent {\bf Case B1b': $\#(\supp\alpha\cap(\A^i\setminus H_0))=2$ and $i\neq0$.} Assume more precisely that
$\supp\alpha\cap(\A^i\setminus H_0)=\{H_1,H_2\}.$
Let $H\in\A^j$ be a line with $a_{H}\neq0$. Then $H$ cannot simultaneously cut $H_1$, $H_2$
if $\alpha \in\R_1(\A)$ (again by the Key Lemma). Assume that $H\parallel H_1$, then $\dfrac{b_{H_2}}{a_{H_2}}=\dfrac{b_{H}}{a_{H}}$.

If $\#\A^j=2$ (in this case $H=H_3$), for $\beta \in \alpha^\perp$, we have $b_{H_0}=\dfrac{b_{H_3}}{a_{H_3}}a_{H_0}$, and this proportion holds for the lines in $\A^i$ which cut $H_3$, therefore, whatever $\sum_{H\in\A^i}a_H$ is, we also obtain $b_{H_1}=\dfrac{b_{H_3}}{a_{H_3}}a_{H_1}$. Consequently, dim$\alpha^\perp=1$, i.e. $\alpha \notin\R_1(\A)$. 

If $\#\A^j>2$, but $\A^j$ doesn't contain any line parallel to $H_2$, then $a_H=0$  for $H\in\A^j\backslash\{H_0,H_3\}$ (use the Key Lemma) and this implies $b_H=0$. Thus,  whatever $\sum_{H\in\A^j}a_H$ is, we have $b_{H_0}= \dfrac{b_{H_3}}{a_{H_3}}a_{H_0}$ and as above, dim$\alpha^\perp=1$. Idem for the case when there is a line $ H_4\in\A^j$ with $H_4 \parallel H_2$ and $a_{H_4}=0$.

Consider now the case when $\#\A^j>2$ and there is a line $ H_4\in\A^j$ with $H_4 \parallel H_2$ and $a_{H_4}\ne 0$.
The Key Lemma implies that $\supp \alpha\subset\{H_0,H_1,H_2,H_3,H_4\}$. If $\beta\in\alpha^\perp$, then again $\supp\beta\subset\{H_0,H_1,H_2,H_3,H_4\}$.

If $\sum_{H\in\A^j}a_H$ or $\sum_{H\in\A^i}a_H$ is non-zero, then $\alpha$ and $\beta$ must be proportional.
On the other hand, if both these two sums equal zero, then the linear subspace $\alpha^\perp$ is given by the following equations:
\begin{eqnarray}
 b_{H_0}+b_{H_1}+b_{H_2}&=&0\\
b_{H_0}+b_{H_3}+b_{H_4}&=&0\\
a_{H_1}b_{H_4}-a_{H_4}b_{H_1}&=&0\\
a_{H_2}b_{H_3}-a_{H_3}b_{H_2}&=&0
\end{eqnarray}
Computing the minors of the associated matrix shows that the solution space of this system is of dimension $>1$ if and only if $a_{H_1}=a_{H_4}$ and $a_{H_2}=a_{H_3}$.
It follows that the parallelogram $\cP=\{H_1,H_2,H_3,H_4\}$ yields some non-trivial elements in the resonance variety $\R_1(\A)$.

\noindent {\bf Case B1b'': $\mbox{{\rm supp}}\alpha\cap(\A^i\setminus H_0)={H_1}$}

One can assume that for every $\A^j$ there exists at most a line different from $H_0$, whose  associated coefficient is non-zero (otherwise, by  changing the role of $\A^i$ into $\A^j$, we return to the previous case).

If supp$\alpha\subset\{H\parallel H_1\}\cup\{H_1\}$: Whence $a_{H_0}\neq0$, $b_H=\dfrac{b_{H_0}}{a_{H_0}}a_H \mbox{ }\forall H\in\A$ so $\alpha\notin\R_1(\A)$. If $a_{H_0}=0$, dim$\alpha^\perp=\#\{$lines of the same direction with $H_1\}$\

If there exists $H_2\in\A^j$ which intersects $H_1$ with $a_{H_2}\neq0$, in order to get rid of the Cases 1 and 3 in the Key Lemma, $a_H$ must be zero for $H\in\A\backslash\{H_0,H_1,H_2\}$. Under this condition, if there is no line in $\A^j$ parallel to $H_1$, the same argument as in Section 1.2 shows that $\alpha$ and $\beta$ are collinear. (idem for $\A^i$ and $H_2$).\

Now we assume that there is a line in $\A^i$ (resp. $\A^j$) parallel to $H_2$ (resp. $H_1$), if $\sum_{H\in\A^i}a_H\neq0$ or $\sum_{H\in\A^j}a_H\neq0$, $\alpha$ and $\beta$ will be collinear. In the opposite case, i.e, $a_{H_0}=-a_{H_1}=-a_{H_2}$, it is easy to check that $\alpha\in\R_1(\A)$. (Note that this case is degenerated from the case 1.2 when $H_2$ plays the role of $H_3$ and $a_{H_2},a_{H_4}$ equal zero.)

\medskip

\noindent {\bf Case B2: $\#\A^i=2$ and $i\neq0$}

Now supp$\alpha\subset\{H|H\cap H_0\mbox{ in a double point}\}\cup\{H_0\}$. Therefore, if $a_{H'}\neq0$ for some $H'\neq H_0$ which cuts $H_1$ at a double point, ($H'$ cuts $H_0$ at a double point), $\alpha\notin\R_1(\A)$.\

If $a_H=0$ for all $H\in\A\backslash\{H_0\}\mbox{ such that }H\not\parallel H_1$, as above, we see that $a_{H_0}=0$ if $\alpha\in\R_1(\A)$. In this case, dim$\alpha^\perp=\#\{$lines of the same direction with $H_1\}$.

\medskip

\noindent {\bf Case B3: $i=0$ ($H_1\parallel H_0$)}
All the lines $H$ which cut $H_0$ must cut $H_1$ at a double point, so $b_H=\dfrac{b_{H_1}}{a_{H_1}}a_H$. If there is $ H_2\notin\A_0$ such that $a_{H_2}\neq0$, then for every $H\in\A^0$, since $H\cap H_2$ is a double point, $b_H=\dfrac{b_{H_2}}{a_{H_2}}a_H=\dfrac{b_{H_1}}{a_{H_1}}a_H$. Besides, whatever the multiplicity of $H_2\cap H_0$ is, we also have $b_{H_0}=\dfrac{b_{H_2}}{a_{H_2}}a_{H_0}$. Thus, $\alpha$ and $\beta$ are collinear.

If $a_H=0$ for all $H\notin\A_0$, it is easy to see that $\alpha^\perp\subset\{\beta~~\vert~~ b_H =0\text{ for all} H\notin\A_0\}$. In that case $\alpha\in\R_1(\A)$ if and only if  $\#\A_0\geq 2$.\

Thus we have proved the following main result.

\begin{thm} \label{t3} (Description of the resonance varieties for the class $\CC_2$)\

We denote by 
$$\A_{k,l,p,q}=\{H_0\}\cup\{H_k,H_l,H_p,H_q|H_k\parallel H_l, H_p\parallel H_q,H_k\cap H_p\cap H_0\neq\emptyset, H_l\cap H_q\cap H_0\neq\emptyset\}$$
the parallelograms in $\C^2$ constructed with the lines in $\A$ and having $H_0$ as a diagonal.
If $\A^0=\emptyset$, then 
$$\R_1(\A)=\bigcup_{i=1}^n\{\alpha|\supp \alpha\subset\A_i\} \cup \bigcup_{\#\A^j>2}\{\alpha|\sum_{H\in\A^j}a_H=0,\supp \alpha\subset\A^j\}\cup$$
$$\cup \bigcup_{\A_{k,l,p,q}}\{\alpha| a_{H_0}+a_{H_k}+a_{H_p}=0, a_{H_k}=a_{H_q},a_{H_l}=a_{H_p},\supp\alpha\subset\A_{k,l,p,q}\}.$$
Otherwise
$$\R_1(\A)=\bigcup_{i=0}^n\{\alpha|\supp \alpha\subset\A_i\} \cup \bigcup_{\#\A^j>2}\{\alpha|\sum_{H\in\A^j}a_H=0,\supp \alpha\subset\A^j\}\cup$$
$$\cup \bigcup_{\A_{k,l,p,q}}\{\alpha| a_{H_0}+a_{H_k}+a_{H_p}=0, a_{H_k}=a_{H_q},a_{H_l}=a_{H_p},\supp\alpha\subset\A_{k,l,p,q}\}.$$

\end{thm}

\begin{rk} \label{r3}
1. All the components of $\R_1(\A)$ except those coming from the parallelograms $\cP=\A_{k,l,p,q}$
(when they exist) are local components. In fact the first ones in the formulas above are obviously local when we consider the associated projective arrangement $\A'$.

2. It is obvious that
$$\dim \{\alpha|\supp \alpha\subset\A_i\}=\#\A_i,~~~~
\dim \{\alpha|\sum_{H\in\A^j}a_H=0,\supp \alpha\subset\A^j\}=\#\A^j -1,$$
and
$$\dim \{\alpha| a_{H_0}+a_{H_k}+a_{H_p}=0, a_{H_k}=a_{H_q},a_{H_l}=a_{H_p},\supp\alpha\subset\A_{k,l,p,q}\}=2.$$
On the other hand, it is known that the resonance varieties  $\R_k(\A)$ enjoy the filtration by dimension property,
namely $\R_k(\A)$ is the union of the irreducible components $E$ of $\R_1(\A)$ with $\dim E >k$, see for instance \cite{DPS1}.
These two facts and the above Theorem yield a complete description of all the resonance varieties  $\R_k(\A)$.

\end{rk}

We can also define the support of a local system $\lambda$, and denote it by $\supp \lambda$, to be the set of lines $H\in\A$ such that the associated monodromy $\lambda_H\neq1$. 
\begin{cor} \label{c7} (Description of the characteristic varieties for the class $\CC_2$)\

If $\A^0=\emptyset$, then 
$$\V_1(\A)=\bigcup_{i=1}^n\{\lambda |\supp \lambda \subset\A_i\}\cup \bigcup_{\#\A^j>2}\{ \lambda | \prod_{H\in\A^j}\lambda_H=1,\supp\lambda \subset\A^j\} \cup$$
$$\cup \bigcup_{\A_{k,l,p,q}}\{ \lambda |   \lambda_{H_0}\lambda_{H_k}\lambda_{H_p}=1, \lambda_{H_k}=\lambda_{H_q},\lambda_{H_l}=\lambda_{H_p},\supp\lambda \subset\A_{k,l,p,q}\}.$$
Otherwise
$$\V_1(\A)=\bigcup_{i=0}^n\{\lambda |\supp \lambda \subset\A_i\}\cup \bigcup_{\#\A^j>2}\{ \lambda | \prod_{H\in\A^j}\lambda_H=1,\supp\lambda \subset\A^j\} \cup$$
$$\cup \bigcup_{\A_{k,l,p,q}}\{ \lambda |   \lambda_{H_0}\lambda_{H_k}\lambda_{H_p}=1, \lambda_{H_k}=\lambda_{H_q},\lambda_{H_l}=\lambda_{H_p},\supp\lambda \subset\A_{k,l,p,q}\}.$$

\end{cor}

\begin{ex} \label{ex3}
Let us consider the
arrangement $\A$ in $\PP^2$ given by the equation $xyz(x-z)(x-y)(y-z)(y
+x-2z)$. After choosing $z=0$ as the line at infinity $H_{\infty}$, the line $x-y$
will take the role of $H_0$ in the above description. Let number the
lines: $H_1:x=0$, $H_2:y=0$, $H_3:x-z=0$, $H_4:y-z=0$, $H_5:x+y-2z=0$
(see the figure below).
\begin{figure}[h]
\begin{center}
\psfragscanon
\includegraphics[scale=0.6]{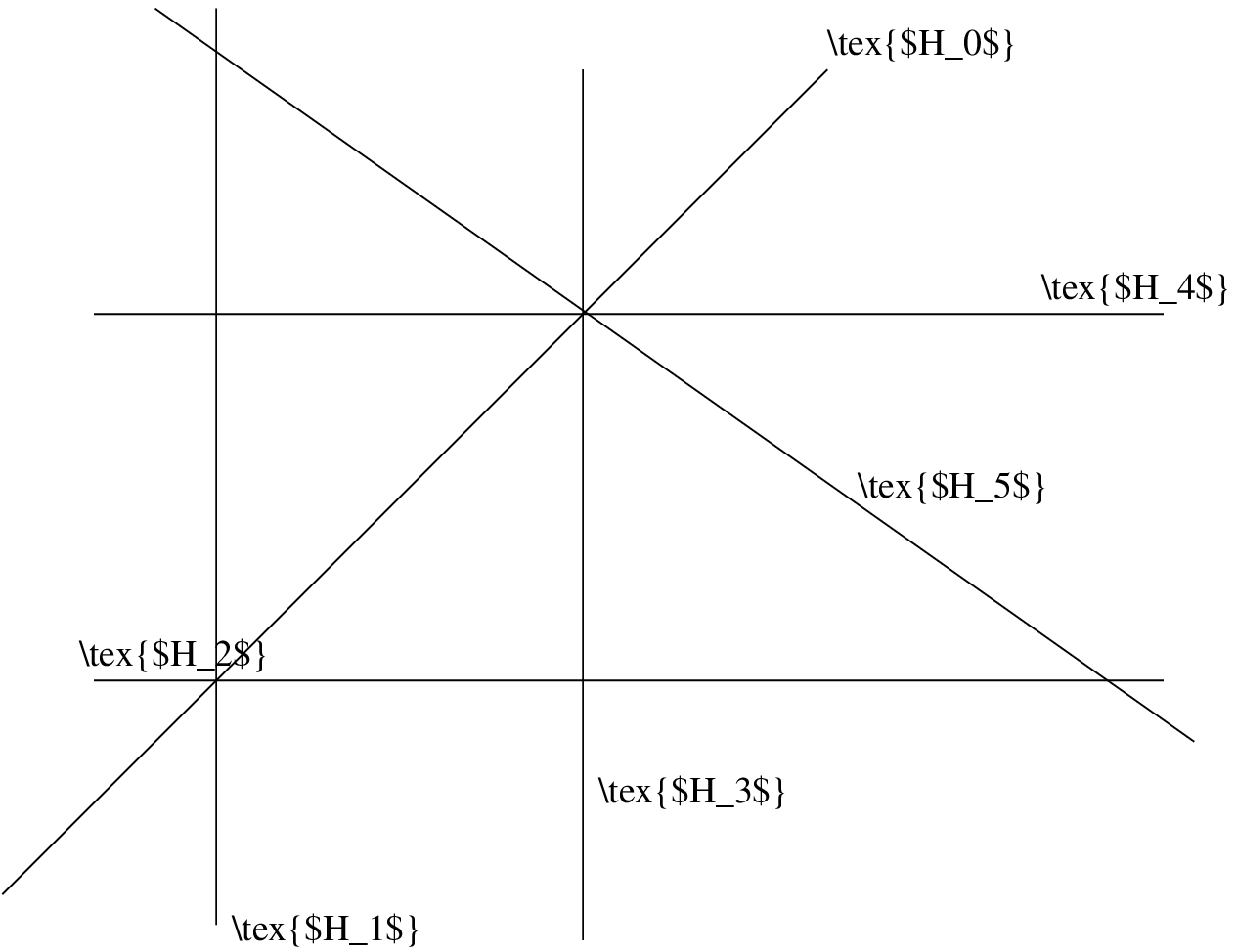}
\end{center}
\end{figure}

By Theorem \ref{t3}, the resonance variety
of $\A$ has the following components:

\noindent $\bullet$ 2 components $E_1$ and $E_2$ corresponding to the two families of parallel lines.

\noindent $\bullet$ 2 components $E_3$ and $E_4$ corresponding  to the central arrangements of cardinal 3
and respectively 4.

\noindent $\bullet$ 1 component $E_5$ corresponding  to the parallelogram determined by the lines $H_1$,
$H_2$, $H_3$ and $H_4$. This component $E_5$ corresponds to a regular
mapping $f_5:M \to \C \setminus \{0,1\}$, see for instance \cite{D2}, given by 
$$f_5(x,y)= \frac{x(y-1)}{y(x-1)}.$$
Note that the fiber over $1 \in \C$ (which is deleted) is precisely the line $H_0$.
If we consider $f_5$ as a pencil of plane curves in $\PP^2$, the corresponding fiber is $H_0 \cup H_{\infty}$, which explains  our remark following Theorem \ref{t0}.

The above components, except $E_4$, are  2-dimensional, $\dim E_4=3$ and they satisfy $E_i \cap E_j=0$ 
for $i \ne j$ (as the general theory predicts).

\end{ex}

\end{document}